\documentclass{llncs}

\usepackage{amsmath,amssymb,amsfonts}
\usepackage{mathtools}
\usepackage{dsfont}
\usepackage[textsize=footnotesize]{todonotes}
\usepackage{enumerate}
\usepackage[breaklinks,colorlinks,citecolor=blue,linkcolor=blue,urlcolor=blue]{hyperref}
\usepackage{cleveref}
\usepackage{orcidlink}
\usepackage{booktabs}

\spnewtheorem{claimnr}{Claim}{\itshape}{\rmfamily}

\newcommand{\R}{\mathds{R}}
\newcommand{\Rp}{\mathds{R}_{\geq 0}}
\newcommand{\Rpp}{\mathds{R}_{> 0}}

\newcommand{\N}{\mathds{N}}

\newcommand{\deltaout}{\delta^{+}}
\newcommand{\deltain}{\delta^{-}}

\newcommand{\abs}[1]{\left\lvert{#1}\right\rvert}

\newcommand{\define}{\coloneqq}
\newcommand{\set}[1]{\left\{#1\right\}}
\newcommand{\suchthat}{\,:\,}

\DeclareMathOperator{\sign}{sign}

\title{Valid Cuts for the Design of\texorpdfstring{\\}{} Potential-based Flow Networks
}

\author{Pascal Börner\inst{1}\orcidlink{0009-0007-0906-8828} \and Max Klimm\inst{2}\orcidlink{0000-0002-9061-2267} \and Annette Lutz\inst{1}\orcidlink{0009-0008-7699-7018} \and \\ Marc E. Pfetsch\inst{1}\orcidlink{0000-0002-0947-7193}  \and Martin Skutella\inst{2}\orcidlink{0000-0002-9814-1703} \and Lea Strubberg\inst{2}\orcidlink{0009-0009-8505-3614} }
\institute{Department of Mathematics, TU Darmstadt, Germany \\
\email{$\{$boerner,lutz,pfetsch$\}$@mathematik.tu-darmstadt.de}
\and
Institute of Mathematics, TU Berlin, Germany\\
\email{$\{$klimm,skutella,strubberg$\}$@math.tu-berlin.de}\\
}

\begin{document}

    \maketitle

    \begin{abstract}
    The construction of a cost minimal network for flows obeying physical laws is an important problem for the design of electricity, water, hydrogen, and natural gas infrastructures.
    We formulate this problem as
    a mixed-integer non-linear program with potential-based
    flows. The  non-convexity of the constraints stemming from the potential-based flow model together with the binary variables indicating the decision to build a connection make these programs challenging to solve.
    We develop a novel class of valid inequalities on the fractional relaxations of the binary variables.
    Further, we show that this class of inequalities can be separated in polynomial time for solutions to a fractional relaxation.
    This makes it possible to incorporate these inequalities into a branch-and-cut framework.
    The advantage of these inequalities is lastly demonstrated in a computational study on the design of real-world gas transport networks. 
    \end{abstract}

\section{Introduction}
\label{sec:Intro}

The availability of networks for electricity, water, hydrogen, natural gas, and traffic is a main driving force for human and economic development. Yet, the installation and maintenance of these infrastructure networks requires long term efforts and vast financial and environmental resources. Thus, for an economical usage of resources, mathematical optimization methods are very important.

This paper studies the mathematical optimization problem of designing a cost-minimal network that satisfies given demands. This problem---sometimes coined the \emph{network synthesis problem}---has been studied extensively for a classic flow model \cite{alvarez-miranda2014,buchheim2011,gomory1962,gusfield1983,talluri1996}.
In this problem, we are given an underlying multi-graph $G = (V,A)$ of possible network connections.
Every potential arc~$a \in A$ has linear costs for installing capacities and the goal is to minimize costs such that a given set of demands can be satisfied.

While obviously important, these models have shortcomings as a model for the construction of physical infrastructure networks for electricity, water, hydrogen, and natural gas since they are based on a simplistic flow model where flow can be routed arbitrarily along the network as long as flow conservation constraints at the nodes are satisfied. In reality, physical flows such as electricity, water, hydrogen, natural gas, and (to a lesser and not immediate extent) individual traffic cannot be routed at will in the network, but instead is governed by physical principles such as resistances, pipe pressures, and minimal travel times.

To address this issue, we adopt the potential-based flow model introduced by Birkhoff and Diaz~\cite{birkhoff1956} and further discussed by Rockafellar~\cite{rockafellar1984}.
In this model, we are given a directed graph $G = (V,A)$ where each arc is characterized by a \emph{resistance} $\beta_a > 0$ that depends on the physical properties of the connection such as electrical resistances of wires, or the roughness of pipes. A flow is a vector $f \in \R^A$ where, for an arc~$a$, a positive flow value $f_a > 0$ indicates a physical (mass) flow in the direction of the orientation of the arc, and a negative flow $f_a < 0$ indicates a (mass) flow against the direction of the orientation of the arc.
The arc flow is governed by a vertex potential vector $\pi \in \R^V$ in the following way. For each arc $a = (u,v)$, the flow on~$a$ satisfies the equation
\begin{align}
\label{eq:weymouth}
\pi_{u} - \pi_{v} = \beta_a \sign(f_a) \abs{f_a}^r,
\end{align}
where the parameter $r > 0$ is called the \emph{degree} of the potential-based flow network. 
We have $r=1$ for DC power networks, $r = 1.852$ for water networks, and $r = 2$ for hydrogen or natural gas networks \cite{Gross19}.
The potential $\pi_v$ on node~$v$ corresponds to voltages in power networks, hydraulic heads in water networks, or squared pressures in hydrogen or natural gas networks.
Due to safety reasons, it is assumed that the node potentials are within given bounds, i.e., that $\pi \in [0,\bar{\pi}]^V$ for a given global maximal potential value $\bar{\pi} > 0$.

For the cost-minimal construction of potential-based flow networks, we are given a directed multi-graph $G = (V,A)$ that represents possible connections. Parallel arcs can be used to allow for connections between two nodes with different properties like diameters, resistances and costs.
We are further given a balance vector $b \in \R^V$ with the property that $\sum_{v \in V} b_v = 0$. 
Our goal is to select a cost-minimal subset of arcs such that there is a potential-based flow satisfying the given balances. That is, we are interested in solving the following mixed-integer non-linear program (MINLP):
\begin{subequations}
\label{problem}
\begin{alignat}{4}
    \min        & \quad&&\sum_{a\in A} c_a \,x_a  &\quad & \\
    \text{s.t.} &      &&\beta_a \sign\left({f_a} \right)\abs{f_a}^r = x_a\, (\pi_u-\pi_v) &\quad & \text{for all } a=(u,v)\in A, \label{problem:weymouth}\\
                &      && \sum_{a\in \deltaout (v)} f_a -   \sum_{a\in \deltain (v)} f_a =b_v  &\quad & \text{for all } v\in V, \label{problem:flow}\\
                &      && f\in  \R^A,\; x\in \set{0,1}^A, & \quad &\\
                &      && \pi \in [0,\bar{\pi}]^V .\label{problem:pressure}
\end{alignat}
\end{subequations}

\paragraph{\bf Our Results.}

The MINLP~\eqref{problem} involves binary decision variables $x \in \{0,1\}^A$ that govern which arcs are constructed, and also features non-linear constraints \eqref{problem:weymouth}. This makes the problem difficult to solve in practice.
One key challenge when solving \eqref{problem} in a branch-and-cut framework is to obtain good linear inequalities.

Our main result is \Cref{thm:main_inequality}, which yields a family of inequalities that are valid for all feasible solutions $(\pi, f, x)$ to \eqref{problem}. 
These inequalities depend on the global maximum pressure bound $\bar{\pi}$ in the network and take the resistances and the degree of the network into account.
We further show that these inequalities can be separated in polynomial time, i.e., for a point $(\pi,f,x)$, we can decide in polynomial time whether all exponentially many inequalities of this type are satisfied, or we can construct one of the inequalities from the family that is violated by $(\pi,f,x)$. 
For the proof of the result, we reduce the problem of finding one of the violated inequalities to the problem of finding disjoint cuts in the network with a suitably chosen total cut value. 
Further using the submodularity of the cut function, we can optimize over this set in polynomial time by using submodular function minimization routines.

Finally, we employ these inequalities and their efficient separation in a spatial branch-and-cut framework.
Our computational results show that the inclusion of these inequalities in the branch-and-cut  process improves the overall computation time.

Some of the previous literature (e.g., Thürauf et al.~\cite{ThueraufGruebelSchmidt24}) deal with individual pressure bounds 
\begin{align}
\label{eq:individual-bounds}
\tag{2e$'$}
\pi_v\in [\bar \pi_v, \underline{\pi}_v] \quad \text{for all $v \in V$}  
\end{align}
instead of the global pressure bound \eqref{problem:pressure}. Setting $\bar \pi \coloneqq \max_{v\in V} \bar \pi_v - \min_{v\in V}  \underline{\pi}_v$, it is straightforward to verify that the MINLP with global pressure bounds $[0,\bar{\pi}]$ is a relaxation of the MINLP with the individual pressure bounds~\eqref{eq:individual-bounds}. Thus, all valid inequalities on $(\pi,f,x)$ for the relaxation remain valid for the restriction.

\paragraph{\bf Related Work.}

Topology optimization has been independently considered for the different applications of potential-based flows. Moreover, for physical networks like gas and water networks, there are mainly two different approaches:

The first approach and also the approach followed in this paper is to construct arcs with minimum cost from a given set of possible connections. This problem has been studied for gas networks~\cite{BorrazSanchezBentBackhausHijaziHentenryck2016,HumFK16,Hum14,HumS17,LiDeySahinidis24,SchweigerLiers2018}, water networks \cite{BragalliDAmbrosioLeeLodiToth12,LiDeySahinidisSusarlaZamarripaDrouven25}, electrical flows~\cite{BarrosAHCPT23} and general potential-based flow networks~\cite{ThueraufGruebelSchmidt24}

The $\mathsf{NP}$-hardness of Problem~\eqref{problem} has been shown explicitly in~\cite{Hum14}  for the degree of $r = 2$; for electrical networks (with a degree of $r = 1$) it follows from the $\mathsf{NP}$-hardness result of the switching problem for DC-networks shown in \cite{LehGhH14}.

The second approach is to determine optimal capacities of the installed arcs corresponding, e.g., to diameters of pipes~\cite{AndreBonnansCornibert2009,AndB11,BobonneauNesterovVial2012,HansenMadsenNielsen1991,WolS96}. 

\section{Preliminaries}
\label{sec:Prelim}

Let $G=(V,A)$ be a directed and weakly connected graph with possibly parallel arcs but without loops.
To clarify notation, we define $\delta^+(v)\subseteq A$ as the set of arcs leaving node $v\in V$, and  $\delta^-(v)\subseteq A$ as the set of arcs entering $v$.
For a set $S\subseteq V$, we denote by $\delta (S) \subseteq A$ the set of arcs that connect a node $u\in S$ with a node $v\notin S$ in either direction.

A \emph{flow} in $G$ is a vector $f \in \mathbb R^A$, where negative flow over an arc indicates a positive flow in the opposite direction.
Let $B\coloneqq \{b\in \R^V \suchthat \sum_{v\in V} b_v = 0\}$ be the set of \emph{balanced vectors}.
We call a flow $f\in \R^A$ a $b$-\emph{transhipment} or the \emph{flow induced by} $b$ with $b\in B$ if it satisfies the flow conservation constraints~\eqref{problem:flow}.
For $s$, $t\in V$ with $s\neq t$, a $b$-transhipment is an $(s,t)$-\emph{flow of value} $d\in \R$ if and only if $b=d\cdot \chi_{s,t}$.
Here, we define the vector $\chi_{s,t} \coloneqq \chi_s-\chi_t \in B $, where $\chi_v$ is the standard unit vector in~$\R^V$  corresponding to component $v$.

We call the tuple $\mathcal{N}=(G,\beta, r)$ with resistances $\beta \in \R^A_{>0}$ and degree $r>0$  a \emph{potential-based flow network}, or simply \emph{network}.
The \emph{conductance} $\mu\in \Rpp^A$ is defined by $\mu_a = 1/\sqrt[r]{\beta_a}$ for $a\in A.$
A \emph{potential-based flow} in a network $\mathcal{N}$ is a pair $(\pi,f)$ of a potential vector $\pi \in \R^V$ and flow vector $f$ for which equation~\eqref{eq:weymouth} holds.

It is an easy observation taht the potential vector is shift-invariant, i.e., for a potential-based flow $(\pi,f)$, the shifted potentials $\pi + \lambda\,\mathbf{1}$ also fulfill~\eqref{eq:weymouth} for all $\lambda \in \R$, where $\mathbf{1}\in \R^V$ is the all ones vector.
This, we will hencefirth assume without loss of generality that all potential-based flows $(\pi,f)$ have their smallest node potential normalized to $0$.
We call a potential-based flow $(\pi,f)$ in a network $\mathcal N$ a \emph{potential-based $b$-transhipment} or a \emph{potential-based $(s,t)$-flow} if $f$ is  a $b$-transhipment or an $(s,t)$-flow in $G$, respectively.

Birkhoff and Diaz \cite{birkhoff1956} showed that for every $b\in B$ there exists a unique potential-based $b$-transhipment $(\pi,f)$. 
This establishes the existence of a function $\smash{\pi^{\mathcal{ N}}\colon B \to \R^V}$ that maps a balanced vector $b \in B$ to potentials $\smash{\pi^{\mathcal N}(b)\in \R^V}$, with a smallest node potential of $0$, and a function  $f^{\mathcal N}\colon B \to \R^A$, mapping $b\in B$ to the flow $f^{\mathcal{N}}(b) \in \R^A$, so that the pair $(\pi^{\mathcal N}(b),f^{\mathcal N}(b))$ fulfills \eqref{problem:flow} and~\eqref{eq:weymouth}.
Note that $\pi^\mathcal N(\lambda b) = \lambda^r  \pi^\mathcal N(b)$ for all $\lambda \in \R$; see also \cite[Observation 3.1]{Gross19}.
In the context of infrastructure networks, the \emph{entry} nodes $\smash{T^+\subseteq V}$ and \emph{exit} nodes $\smash{T^-\subseteq V}$, with $\smash{T^+\cap T^-=\emptyset}$, are typically fixed.
We refer to their union $\smash{T\coloneqq T^+\cup T^-}$ collectively as \emph{terminals}. As a demand on these terminals we consider only those that are balanced.
Therefore, let 
\[
B(T)\define \{b\in B \suchthat 
b_v \geq 0\; \forall\, v \in T^+,\; b_v\leq 0\; \forall\, v \in T^-,\; b_v = 0\; \forall\, v \notin T\}.
\]
For $X\subseteq T$ and $b\in B(T)$, we define $b(X)\coloneqq \sum_{v\in X} b_v$.


\subsection{Effective Resistance and Conductance}

For a potential-based flow network and two designated nodes $s$ and $t$, the \emph{effective resistance}~$R^\mathcal N_{s,t}$ is the potential difference required to send one unit of flow from $s$ to $t$, i.e., $$\smash{R^\mathcal N_{s,t} \define \pi_s^\mathcal N (\chi_{s,t})-  \pi_t^\mathcal N(\chi_{s,t})}.$$
We define the \emph{effective conductance} between $s$ and $t$ as
\(U^\mathcal N_{s,t} = (R^\mathcal N_{s,t} )^{-1/r}\).
Two potential-based flow networks $\mathcal N$  on $G=(V,A)$ and $\mathcal N'$ on $G'=(V',A')$ of degree $r$ are called \emph{equivalent} with respect to a set of terminal nodes $T\subseteq V\cap V'$ if $\pi_v^\mathcal N(b) = \pi_v^{\mathcal N'}(b)$ for every $v\in T$ and $b\in B(T)$. The following result shows that every two-terminal network is equivalent to a network consisting of a single arc.

 \begin{theorem}[\cite{Gross19,Reduction23}]
    \label{thm:reduction}
    A network $\mathcal N=(G,\beta,r)$ with terminal nodes $T=\{s,t\}\subseteq V$ is equivalent to the network $((T,\{(s,t)\}),\beta',r)$ with
    $\beta'_{(s,t)} = R^\mathcal{N}_{s,t}>0$.
\end{theorem}

The following lemma states another useful property of two-terminal networks.

\begin{lemma}[{\cite[Lemma 9]{Reduction23}}]
    \label{lem:s-t-flows}
    Let $s$, $t\in V$ be two distinct nodes in a network~$\mathcal{N}$ and $d>0$.
    Let $\pi \define \pi^\mathcal N(d\cdot \chi_{s,t})\in \R^V$ be the induced potentials of the potential-based  $(s,t)$-flow  of value $d$ in $\mathcal N$.
    Then, $\pi_s \geq \pi_v \geq \pi_t$  for all $v\in V$.
    Furthermore, $\pi_s-\pi_t = d^r \cdot R_{s,t}^\mathcal N$.
\end{lemma}

The following results show how effective resistance and conductance behave under different network operations.

\begin{proposition}[\cite{Reduction23,Raber22}]
\label{prop:contracting_deleting_arc}
    Let  $\mathcal N=(G,\beta,r)$ be a potential-based flow network and let $s$, $t \in  V$. Then, the following hold:
    \begin{enumerate}[(i)]
        \item \label{it:prop-deleting} If $\mathcal N'$ is a network resulting from $\mathcal N$ by
    either increasing the resistance of an arbitrary arc or by
    deleting an arc, then 
    \[R_{s,t}^{\mathcal N'} \geq R_{s,t}^{\mathcal N} \quad \text{ and } \quad U_{s,t}^{\mathcal N'} \leq U_{s,t}^{\mathcal N}. \]
    \item \label{it:prop-contracting} If $\mathcal N''$ is a network resulting from $\mathcal N$ by either decreasing the resistance of an arbitrary arc or by contracting two nodes, then
    \[R_{s,t}^{\mathcal N''} \leq R_{s,t}^{\mathcal N} \quad \text{ and } \quad U_{s,t}^{\mathcal N''} \geq U_{s,t}^{\mathcal N}. \]
    \end{enumerate}   
\end{proposition}

\subsection{Graph Cuts}

In this paper, we present valid inequalities for~\eqref{problem}, which are based on cuts inside the network.
A \emph{cut} in a graph $G$ is a partition of the node set $V$ into two disjoint subsets $S\subseteq V$ and $V\setminus S\subseteq V$.
We call the arc set $\delta (S)\subseteq A$ the set of \textit{arcs crossing cut} $S$. 
A cut $S$ is an $(X,Y)$-\textit{cut} if $X\subseteq S$ and $Y\subseteq V\setminus S$.
If $X=\{s\}$ and $Y=\{t\}$, we change the notation slightly and call it an $(s,t)$-cut. 

For a weighted graph $G=(V,A,w)$ with non-negative arc weights $w\in \Rp^A$, we define the \emph{cut function} $z$ that maps a cut to the total weight of arcs crossing the cut, i.e., $z\colon 2^V\rightarrow \Rp$, $S\mapsto \sum_{a\in \delta(S)} w_a$, where $2^V$ is the power set of $V$.
The function $z$ is submodular (see, e.g., \cite{mccormick}), that is, 
\begin{equation}
    \label{eq:submodular}
    z(A\cup B) + z(A\cap B)  \leq z(A) +z(B) \quad \text{for all } A,B\subseteq V.
\end{equation}
Submodular functions can be minimized in polynomial time; see, e.g.,~\cite{mccormick}. 
More generally, a minimum of a submodular function defined on a subset $\mathcal X \subseteq 2^V$ closed under union and intersection can be computed in polynomial time.

We say that $S_1,\dots, S_k\in 2 ^V$ are $k$ disjoint $(X,Y)$-cuts if  $S_i$ is an $(X,Y)$-cut and $\delta(S_i)\cap \delta(S_j) =\emptyset$ for all $i\neq j$.
The $k$ disjoint $(X,Y)$-cut problem is the problem of finding $k$ disjoint $(X,Y)$-cuts such that the sum of the weight of the arcs crossing the cuts is minimal over all such cuts.
This problem is a slight generalization of the $k$ disjoint $(s,t)$-cut problem studied in~\cite{Wagner1990}.
In particular, for non-negative arc weights, it is shown in~\cite{Wagner1990} that the problem has a compact LP formulation and can thus be solved in polynomial time.


\section{Inequalities for Potential-based Flows}
In the following we state valid inequalities that must hold if there exists a potential-based flow satisfying global pressure bounds.
We start with an observation on rather simple networks defined on a path with parallel arcs.

\subsection{$(s,t)$-Flows on Multi-Paths}

In this subsection, we consider a multi-graph~$G=(V,A)$ with $k+1$ nodes, $V=\set{v_0,v_1,\dots, v_{k}}$, where only arcs between consecutive nodes $v_i$ and $v_{i+1}$ are allowed.
We denote by $A_i$ the set of arcs connecting node $v_{i-1}$ to $v_{i}$ such that $A=\cup_{i=1}^k A_i.$
We call such a graph \emph{multi-path of length~$k$}.
Additionally, we define a \emph{multi-path network of length $k$} as a network with an underlying multi-path of corresponding length.

Given such a multi-path network $\mathcal N$ with conductance vector $\mu\in \R^A_{>0}$, we introduce variables~$u_i\in \R_{>0}^k$ for the sum of the conductances of the arcs between the nodes~$v_{i-1}$ and $v_i$ for $i = 1, \dots, k$, so that
\begin{equation}
    \label{eq:def_u_i}
    u_i \define \sum_{a\in A_i} \mu_a.
\end{equation}

This first lemma gives a closed form expression for the effective conductance of multi-path networks.
\begin{lemma}
    \label{lem:eff_con_onion}
    Let $\mathcal N$ be a multi-path network of length $k$ and degree $r>0$, let $s =v_0$ and $t = v_k$.
    The effective conductance between $s$ and $t$ in $\mathcal{N}$ is given by
    \begin{equation}
        \label{eq:eff_con}
        U_{s,t}^\mathcal N = (\textstyle \sum_{i=1}^k u_i^{-r})^{-1/r}.
    \end{equation}
\end{lemma}

\begin{proof}
    It is shown in~\cite[Lemmas 6 \& 7]{Gross19} that the conductance of parallel arcs sums up in the same way as the resistance of arcs in series.
    We can therefore reduce all parallel arcs in $A_i$ to one arc $a_i= (v_{i-1},v_{i})$ with conductance $u_i$ without changing the behavior of the potentials at $s$ and $t$.
    The resistance of these arcs is then given by $\beta_{a_i} = u_i^{-r}$.
    We further reduce the network to one arc $a=(s,t)$ with resistance $\beta_a = \sum_{i=1}^k \beta_{a_i}$ again without changing the behavior of the potentials at $s$ and $t$. 
    Let $\mathcal N'=((\set{s,t},\set{(s,t)}),(\beta_a),r)$ be the resulting network. 
    Because $\mathcal N$ and $\mathcal N'$ are equivalent with respect to the node set $\set{s,t}$, we have $R_{s,t}^\mathcal N = R_{s,t}^{\mathcal N'}  = \beta_a$ according to \Cref{thm:reduction}.
    Using the definition of effective conductance, we obtain
    \[
    U_{s,t}^\mathcal N  = \frac{1}{\sqrt[r]{\beta_a}} = (\textstyle \sum_{i=1}^k \beta_{a_i})^{-1/r} =  (\textstyle \sum_{i=1}^k u_i^{-r})^{-1/r}.\eqno\qed
    \]
\end{proof}

We proceed to use this closed formula to obtain a lower bound on the effective conductance of a network that supports an $(s, t)$-flow of value $d$. 

\begin{lemma}
\label{lem:n_onions}
    Let $\mathcal N$ be a multi-path network of degree $r>0$ and length $k$, let $s =v_0$ and $t = v_k$.
    Let $(\pi,f)$ be a potential-based $(s,t)$-flow of value $d>0$.
    If $\pi\in [0,\bar{\pi}]^V$, then
    \[  U_{s,t}^\mathcal N   \geq \frac{d}{\sqrt[r]{\bar{\pi}}}. \]
\end{lemma}

\begin{proof}
    Because $(\pi,f)$ is the potential-based flow in $\mathcal N$ induced by $b = d\cdot \chi_{s,t}$, it follows from \Cref{lem:s-t-flows} that
    \begin{align*}
        \pi_s - \pi_t  
        = \pi^{\mathcal{N}}_s(d\cdot \chi_{s,t}) - \pi^{\mathcal{N}}_t(d\cdot \chi_{s,t}) 
        = d^r R_{s,t}^{\mathcal{N}},
    \end{align*}
    which is equivalent to $\sqrt[r]{\pi_s - \pi_t} \cdot U_{s,t}^{\mathcal{N}} = d$.
    Since $\bar{\pi} \geq \pi_s - \pi_t$, we obtain
    \[ 
    \sqrt[r]{\bar{\pi}} \cdot U_{s,t}^{\mathcal{N}} \geq \sqrt[r]{\pi_s - \pi_t } \cdot U_{s,t}^{\mathcal{N}} = d.
    \eqno\qed
    \]
\end{proof}

In the following we state an inequality for potential-based $(s,t)$-flows with global pressure bounds $[0,\bar{\pi}]$ which is linear in the conductance $\mu$. 
Intuitively, it states that the average flow that can be sent from node $v_i$ to node $v_{i+1}$, when giving these nodes a potential difference of $\bar{\pi}/k$, must be at least the value of a maximum potential-based $(s,t)$-flow;
see \Cref{fig:2onions} for an example.

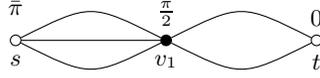
\begin{figure}[tb]
\tikzstyle{nodeL}=[circle,draw, minimum size=4pt, inner sep = 0]
\tikzstyle{nodeI}=[circle,draw, minimum size=4pt, inner sep = 0, fill]
\tikzstyle{edge}=[]
   \centering
        \begin{tikzpicture}
                \node (dummy) at (0,1.2) {};
                \node [nodeL, label={below:$s$}  ] (s) at (-2,0) {};
                \node [label={above:$\bar{\pi}$}, yshift=3pt] at (-2,0) {};
                \node [nodeI, label={below:$v_1$}, label={above:$\tfrac{\bar{\pi}}{2}$}](n) at (0,0) {};
                \node [nodeL, label={below:$t$}, label={above:$0$}] (t) at (2,0) {} ;

                \draw [edge] 
                (s) -- (n);
                
                \draw[edge] 
                (s) .. controls (-1,0.5) .. (n)
                (n) .. controls (1,0.5) .. (t);   

                \draw[edge] 
                (s) .. controls (-1,-0.5) .. (n)
                (n) .. controls (1,-0.5) .. (t);  
        \end{tikzpicture}
\caption{Assuming $\mu = \mathbf{1}$, we can send a flow of $3\sqrt[r]{\bar{\pi}/2}$ over $A_1$, and  $2\sqrt[r]{\bar{\pi}/2}$ over $A_2$. Thus, by \Cref{lem:eff_cond_inequality} a feasible potential-based $(s,t)$-flow in the depicted network has at most a value of $5/2 \sqrt[r]{\bar{\pi}/2}$.}
\label{fig:2onions}
\end{figure}

\begin{theorem}
    \label{lem:eff_cond_inequality}
    Let $\mathcal N$ be a multi-path network of length $k$ and degree $r>0$, let $s =v_0$ and $t = v_k$.
    Let $(\pi,f)$ be a potential-based $(s,t)$-flow of value $d>0$.
    If $\pi\in [0,\bar{\pi}]^V$, then
    \begin{equation*}
        \frac{1}{k\sqrt[r]{k}}\sum_{i=1}^k \sum_{a\in A_i} \mu_a \geq \frac{d}{\sqrt[r]{\bar{\pi}}}.
    \end{equation*}
\end{theorem}

\begin{proof}
    For the proof we use \Cref{lem:n_onions} and claim that
    \begin{equation}
        \label{eq:eff_conductance_bound}
        \frac{1}{k\sqrt[r]{k}} {\sum_{i=1}^k \sum_{a\in A_i} \mu_a}  
        \geq  U^\mathcal N_{s,t}.
    \end{equation}
    Note that~\eqref{eq:eff_conductance_bound} is by Equations~\eqref{eq:def_u_i} and~\eqref{eq:eff_con} equivalent to 
    \begin{equation}
         \label{eq:eff_conductance_bound_equivalent}
        \frac{1}{k\sqrt[r]{k}}\sum_{i=1}^k u_i 
        \geq  \Big(\sum_{i=1}^k u_i^{-r}\Big)^{-1/r}.
    \end{equation}
    We show that~\eqref{eq:eff_conductance_bound_equivalent} holds for any positive numbers $u_1,\dots,u_n \in \Rpp$ by applying the inequality of arithmetic and geometric mean twice, once for $u_1,\dots, u_k$ and once for $u_1^{-r},\dots, u_k^{-r}$, so that
    \begin{align*}
        \textstyle \Big(\sum_{i=1}^k u_i \Big)\Big(\sum_{i=1}^k u_i^{-r} \Big)^{1/r}
        & \textstyle \geq k \Big(\prod_{i=1}^k u_i \Big)^{1/k}  \Big(k \Big( \prod_{i=1}^k u_i^{-r} \Big)^{1/k}\Big)^{1/r}\\
        & \textstyle =k \sqrt[r]{k} \Big(\prod_{i=1}^k u_i \big( u_i^{-r}\big)^{1/r}\Big)^{1/k}
        = k \sqrt[r]{k}.&\hfill\qed
    \end{align*}
\end{proof}

\subsection{General $(s,t)$-Flows}
For $(s,t)$-flows in general networks, \Cref{lem:eff_cond_inequality} extends in the following way.

\begin{theorem}
    \label{thm:n-cuts_inequality}
    Let $\mathcal N=(G,\beta, r)$ be a potential-based flow network and let $s$, $t\in V$.
    If there exists a potential-based $(s,t)$-flow of value $d > 0$ in $\mathcal N$ with potentials $\pi \in [0,\bar{\pi}]^V$,
    then
    \begin{equation}
        \label{eq:st-cut}
        \frac{1}{k\sqrt[r]{k}}\sum_{i=1}^k \sum_{a\in \delta(S_i)} \mu_a \geq \frac{d}{\sqrt[r]{\bar{\pi}}}
    \end{equation}
    holds for all $k$ disjoint $(s,t)$-cuts $ S_1\subseteq \dots \subseteq S_{k} \subseteq  V$.
\end{theorem}

\begin{proof}
    By contracting the nodes in $S_i\setminus S_{i-1}$ for each $i=1,\dots,k+1$, where we set $S_0\coloneqq \emptyset$ and $S_k \coloneqq V$, we obtain a multi-path network $\mathcal N'$ of length $k$, so that $A_i = \delta(S_i)$.
    Note that the effective resistance between $s$ and $t$ does not increase in $\mathcal N'$ by \Cref{prop:contracting_deleting_arc}\eqref{it:prop-contracting}, i.e., $ R^{\mathcal N'}_{s,t} \leq R^\mathcal N_{s,t} $.
    
    Consider the potential $\pi'\coloneqq \pi^{\mathcal N}(d\cdot \chi_{s,t})$ of the potential-based $(s,t)$-flow of value $d$ in $\mathcal N'$, with a fixed potential of~$0$ at node~$t$.
    Since the highest potential is attained at $s$ and the lowest at $t$,
    according to \Cref{lem:s-t-flows} it holds for every node~$v_i$ in the multi-path that 
    \begin{equation*}
        0 \leq \pi'_{v_i} \leq \pi'_s -\pi'_t = d^r R^{\mathcal N'}_{s,t} \leq d^r R^\mathcal N_{s,t} = \pi_s -\pi_t \leq \bar{\pi}.
    \end{equation*}
    Applying \Cref{lem:eff_cond_inequality} to $\mathcal N'$ yields the desired inequality.
\qed\end{proof}

\subsection{$b$-Transshipments}
Before we state our main result, we note that every potential-based $b$-trans\-ship\-ment satisfying global pressure bounds 
yields, for any subset $X\subseteq T$ of terminal nodes, a valid potential-based flow 
with sources in $X$ and sinks not in $X$.

\begin{theorem}
\label{thm:b-transhipment}
Let $\mathcal N=(G,\beta,r)$ be a potential-based flow network with balance vector~$b\in B(T)$,
and~$f$ a potential-based $b$-transshipment with corresponding node potentials $\pi\in[0,\bar\pi]^V$.
Then, for any~$X\subseteq T$ with~$b(X)\geq0$, there is a potential-based $b'$-transshipment~$f'$ with corresponding potentials~$\pi'\in[0,\bar\pi]^V$, and
\begin{align}
\{v\suchthat b'_v>0\}\subseteq T^+\cap X,\quad\{v\suchthat b'_v<0\}\subseteq T^-\setminus X,\quad\sum_{v\in T^+\cap X}b'_v=b(X).
\label{eq:b'-conditions}
\end{align}
\end{theorem}

\begin{proof}
We assume w.l.o.g.\ that the arcs in~$A$ are oriented such that~$f_a\geq0$ for all~$a\in A$.
Flow conservation implies that 
\begin{align}
\sum_{a\in\delta^+(S)}f_a\geq b(X)\qquad\text{for each~$S\subseteq V$ with~$X\cap T^+\subseteq S\subseteq V\setminus (T^-\setminus X)$.}
\label{eq:cut-condition}
\end{align}
If we interpret the flow value~$f_a$ as a capacity of arc~$a\in A$, then~\eqref{eq:cut-condition} and the max-flow min-cut theorem imply that there exists a feasible $\hat{b}$-transshipment~$\hat{f}$ in~$G$ with~$0\leq\hat{f_a}\leq f_a$ for all~$a\in A$ and~$\hat{b}\in\R^V$ satisfying the exact same conditions stated for~$b'$ in~\eqref{eq:b'-conditions}.

We can now interpret~$\hat{f}$ as a potential-based flow with the same node potentials~$\pi\in[0,\bar{\pi}]^V$ as~$f$ in network~$\hat{\mathcal N}\coloneqq(G,\hat{\beta},r)$ with increased arc resistances~$\hat{\beta}\geq\beta$: simply let
\[
\hat{\beta}_a\coloneqq\frac{\pi_v-\pi_w}{{\hat{f_a}}^r}\geq\frac{\pi_v-\pi_w}{{f_a}^r}=\beta_a\qquad\text{for all~$a=(v,w)\in A$.}
\]
In the case that~$\hat{f_a}=0$ for some arc~$a\in A$, the resistance~$\hat{\beta}_a$ is infinite, which is equivalent to removing the arc from the network.

We further modify the network by adding a super-source node~$s$ with arcs from~$s$ to all source nodes~$v\in V$ with~$\hat{b}_v>0$, as well as a super-sink node~$t$ with arcs from all sink nodes~$v\in V$ with~$\hat{b}_v<0$ to~$t$. 
We extend potentials~$\pi\in[0,\bar{\pi}]^V$ to potentials on~$V\cup\{s,t\}$ by setting $\pi_s\coloneqq\bar{\pi}$ and~$\pi_t\coloneqq0$.
Moreover, we set the resistances of arcs leaving~$s$ and of arcs entering~$t$ in such a way that the node potentials~$\pi$ induce an~$(s,t)$-flow of value~$b(X)$ in the extended network. 
More precisely, we set
\begin{align*}
\hat{\beta}_{(s,v)}&:=\frac{\pi_s-\pi_v}{{\hat{b}_v}^r}&&\text{for all~$v\in V$ with~$\hat{b}_v>0$, \quad \text{and}}\\
\hat{\beta}_{(v,t)}&:=\frac{\pi_v-\pi_t}{-{\hat{b}_v}^r}&&\text{for all~$v\in V$ with~$\hat{b}_v<0$.}
\end{align*}
If the resistance of an arc~$(s,v)$ or~$(v,t)$ is set to zero, this is equivalent to merging node~$v$ with~$s$ or~$t$, respectively, as both have the same potential. If we now decrease the arc resistances of arcs~$a\in A$ from~$\hat{\beta}_a$ to their original values~$\beta_a$, then, by \Cref{prop:contracting_deleting_arc},
there still exists a potential-based $(s,t)$-flow~$f'$ of the same value~$b(X)$ with corresponding node potentials~$\pi'\in[0,\bar{\pi}]^{V\cup\{s,t\}}$. Restricting~$f'$ to the original network, that is, deleting nodes~$s$ and~$t$ together with their incident arcs, yields the desired~$b'$-transshipment satisfying~\eqref{eq:b'-conditions}.
\qed
\end{proof}

\section{Inequalities for the Topology Optimization Problem}

After having established several inequalities for potential-based flows with global pressure bound $\bar{\pi}$, we are now returning to the topology optimization problem~\ref{problem} and present our main result.

Given an instance $(\mathcal N,r,b,\bar{\pi},c)$ of~\eqref{problem}, we consider a feasible solution $(\pi,f,x)$.
We state a linear inequality in $x$, which can then directly be added in a spatial branch-and-bound algorithm to solve~\eqref{problem}, see also Section~\ref{sec:ComputationalResults}.

\begin{definition}
    \label{def:Nx}
    Let $\mathcal N=(G,\beta,r)$ be a potential-based flow network. Let~$x\in [0,1]^A$ be given. 
    We call $\mathcal N^x\define (G^x,\beta^x,r)$ the network induced by~$x$,
    where we define the graph $G^x\define (V,A^x)$ with arc set~$A^x \define \set{a\in A \suchthat x_a \neq 0}$.
    For the resistance, we set $\beta^x_a\define \beta_a /x_a$ for $a\in A^x$ and denote the corresponding conductance with~$\mu^x$. 
\end{definition}
Note that by the definition of conductance, we obtain:
\begin{equation}
    \label{eq:Nx_conductance}
    \mu^x_a = \sqrt[r]{x_a}\,\mu_a \quad \text{for }  a\in A^x.
\end{equation}

\begin{remark}
    \label{rem:121-Nx}
    There is a one-to-one correspondence between a feasible solution to~\eqref{problem} and the potential-based flow in $\mathcal N^x$.
    If $(\pi,f,x)$ is feasible for~\eqref{problem}, then $(\pi,f^x)$, with  $f^x_a=f_a$ for $a\in A^x$, is exactly the potential-based flow induced by~$b$ in the network $\mathcal N^x$.    
    Conversely, if $(\pi,f')$ is a potential-based $b$-transhipment in $\mathcal N^x$ with $x\in [0,1]^A$, then $(\pi,f,x)$, with $f_a = f'_a$ for $a\in A^x$ and $f_a =0$ otherwise, is feasible for~\eqref{problem} if and only if $\pi_v\in [0,\bar{\pi}]$. 
\end{remark}

\begin{theorem}
    \label{thm:main_inequality}
    Let $(\mathcal N,r, b,\bar{\pi},c)$ be an instance of~\eqref{problem}.
    If $(\pi,f,x)$ is feasible, then
    \begin{equation}
        \label{b-cut}
        \frac{1}{k\sqrt[r]{k}} \sum_{i=1}^k \sum_{a\in\delta(S_i)} \mu_a\, x_a \geq \frac{b(X)}{\sqrt[r]{\bar{\pi}}}
    \end{equation}
    for all subsets $X\subseteq T $ and all $k$ disjoint $(X\cap T^+,T^-\setminus X)$-cuts $S_1\subseteq \dots \subseteq S_k\subseteq V$.
\end{theorem}
    
\begin{proof}
    Since all variables are non-negative, the inequality always holds if $b(X)$ is non-positive.
    We therefore assume, without loss of generality, that $b(X)>0$.
    
    Consider the network  $\mathcal N^x$ induced by $x$ from $\mathcal N$ (\Cref{def:Nx}).
    According to \Cref{rem:121-Nx}, $(\pi,f)$ is a potential-based $b$-transshipment in $\mathcal N^x$, satisfying the given pressure bounds if we ignore the flow over arcs outside the support of~$x$.
    By \Cref{thm:b-transhipment}, network $\mathcal N^x$ admits a potential-based $b'$-transhipment with corresponding potentials $\pi'\in [0,\bar{\pi}]^V$ such that all source nodes of $b'$ are contained in $T^+\cap X$ and the sinks contained in $T^-\setminus X$. 
    As shown in the theorem's proof, we can modify the network $\mathcal N^x$ by adding a super-source~$s$ with arcs to all sources of $b'$ and a super-sink node $t$ with arcs connecting all sinks of $b'$ to $t$.
    
    The resistances of the arcs added can be chosen so that the potentials $\pi'$ extend to $V\cup \{s,t\}$.
    This can be done in such a way that they induce a potential-based $(s,t)$-flow of value $b(X)$, while still satisfying the given pressure bounds.

    Note that the arcs crossing the $k$ disjoint $(s,t)$-cuts $S_1\cup \{s\}\subseteq \dots \subseteq S_k\cup \{s\}$ in the modified network are exactly those crossing the $(X\cap T^+,T^-\setminus X)$-cuts  $S_1, \dots, S_k$ in $\mathcal N^x$. 
    Applying \Cref{thm:n-cuts_inequality} together with~\eqref{eq:Nx_conductance} yields:
    \begin{equation}
        \label{eq:nonlinear_cut}
        \frac{1}{k\sqrt[r]{k}}\sum_{i=1}^k \sum_{a\in \delta(S_i)\cap A^x} \mu_a\, \sqrt[r]{x_a} \geq \frac{d}{\sqrt[r]{\bar{\pi}}},
    \end{equation}
    To conclude the proof, note that $\sqrt[r]{x_a}$ can be replaced by $x_a \in \{0,1\}$.
\qed
\end{proof}

Inequality~\eqref{eq:nonlinear_cut} still holds for non-binary $x\in [0,1]^A$.  
However, expression~\eqref{b-cut} is stronger if $r > 1$.

\begin{theorem}
    \label{thm:separation}
    Inequalities of type~\eqref{b-cut} can be separated for given $x\in[0,1]^A$ in polynomial time.
\end{theorem}

\begin{proof}
    To simplify notation, we say that $S_1, \dots, S_k \in 2^V$ are \emph{disjoint $X$-cuts} for $X \subseteq T$ if $S_1\subseteq \dots \subseteq S_k$ are $k$ disjoint $(X\cap T^+,T^-\setminus X)$-cuts.
     We show that $X\subseteq T$  minimizing the following expression
    \begin{equation}
        \label{expression}
         \frac{1}{k\sqrt[r]{k}} \sum_{i=1}^k \sum_{a\in\delta(S_i)} \mu_a x_a - \frac{b(X)}{\sqrt[r]{\overline{\pi}}}
    \end{equation}
    for all $k$ disjoint $X$-cuts can be found in polynomial time.
    We define the family $\mathcal X _k \define \{X\subseteq T \suchthat \text{there are disjoint $X$-cuts } S_1, \dots, S_k \in 2^V\}$ for every $k \in \N$.
    
    In \Cref{app:claim1}, we show that $\mathcal X_k\subseteq 2^T$ is closed under set union and intersection; see Lemma~\ref{lem:lattice} which is stated and proven there.
    On $\mathcal X_k$ we then define the function $\sigma_k \colon \mathcal X _k \rightarrow \R_{\geq 0}$ which maps $X\in  \mathcal X _k$ to the value
    \begin{equation*}
        \sigma_k(X) \define \min \Bigg\{ \frac{1}{k\sqrt[r]{k}} \sum_{i=1}^k \sum_{a\in\delta(S_i)} \mu_a x_a \suchthat
        \text{disjoint $X$-cuts } S_1, \dots, S_k \subseteq V 
        \Bigg \}.
    \end{equation*}

    \begin{claimnr}
        \label{claim:sigma_k}
        The set function $\sigma_k$ is submodular. 
    \end{claimnr}
    \begin{proof}[of Claim~\ref{claim:sigma_k}]
    \renewcommand\qed{\hfill\rotatebox{90}{$\triangle$}}
        Let $X$, $Y\in \mathcal X_k$, and $S^X_1, \dots, S^X_k $ and  $S^Y_1, \dots , S^Y_k$ the respective minimal, $k$ disjoint $X$-cuts and $Y$-cuts.
        For the submodularity of $\sigma_k$, we first use the definition of $\sigma_k$ and then the submodularity of the cut function, i.e., Inequality~\eqref{eq:submodular}.
        The final step follows from the fact that $S_1^X\cup S_1^Y,\dots ,S_k^X\cup S_k^Y$ are disjoint $X\cup Y$-cuts and $S_1^X\cap S_1^Y,\dots, S_k^X\cap S_k^Y$ disjoint $X\cap Y$-cuts.
        Both statements follow from the proof of Lemma~\ref{lem:lattice} in \Cref{app:claim1}.
        Formally, 
        \begin{align*}
            \sigma_k(X) + \sigma_k(Y) 
             & =  \frac{1}{k\sqrt[r]{k}} \sum_{i=1}^k \bigg( \sum_{a\in\delta(S^X_i)} \mu_a x_a + \sum_{a\in\delta(S^Y_i)} \mu_a x_a \bigg) \\
             &\geq  \frac{1}{k\sqrt[r]{k}} \sum_{i=1}^k \bigg( \sum_{a\in\delta(S^X_i\cup S^Y_i)}  \mu_a x_a + \sum_{a\in\delta(S^X_i\cap S^Y_i)} \mu_a x_a \bigg) \\
             &\geq   \sigma_k(X\cup Y) + \sigma_k(X \cap Y). &\hfill\qed
        \end{align*}
    \end{proof}

    \begin{claimnr}
        \label{claim:sigma_eval_poly}
        The set function $\sigma_k$ can be evaluated in polynomial time. 
    \end{claimnr}
    \begin{proof}[of Claim~\ref{claim:sigma_eval_poly}]
    \renewcommand\qed{\hfill\rotatebox{90}{$\triangle$}}
        We show that we can decide if $X\in \mathcal X_k$ and, if this is the case, also evaluate $\sigma_k(X)$ in polynomial time. 
        We consider the underlying undirected graph $G=(V,E,w)$ of the given network $\mathcal N$ with edge weights given by $w_e=\mu_e x_e\geq 0$.
        We then add nodes $s$ and $t$ and connect $s$ to all nodes in $X\cap T^+$ and $t$ to all nodes in $T^-\setminus X$ where we set the weight of all these edges to infinity.
        We solve the $k$ disjoint $(s,t)$-cut problem by solving an LP formulation stated in~\eqref{min_k_cut} in \Cref{app:claim2} to obtain the value $k\sqrt[r]{k}\, \sigma_k(X)$.
    \qed\end{proof}
    Note that $b$ as a function on $2^T$ is modular.
    This implies together with Claim~\ref{claim:sigma_k} that the function $g_k\colon \mathcal X_k\rightarrow \Rp$, $X\mapsto \sigma_k(X) -b(X)/\sqrt[r]{\bar{\pi}}$ is submodular. 
    Thus, the expression~\eqref{expression} can be minimized by computing the minimum of $g_k$ for every~$k\in \mathbb N$.
    There exist at most $\abs{V}-1$ disjoint $(X\cap T^+,T^-\setminus X)$-cuts.
\qed \end{proof}

In the proof of Claim~\ref{claim:sigma_k}, we show that the functions $\sigma_k$ are submodular for all $k \in \N$. One may wonder whether the function $\sigma \colon 2^T \to \R_{\geq 0}$ defined by
\begin{equation*}
    \sigma(X) \define \min \Bigg\{ \frac{1}{k\sqrt[r]{k}} \sum_{i=1}^k \sum_{a\in\delta(S_i)}  \mu_a x_a  \suchthat
      \text{disjoint $X$-cuts } S_1, \dots, S_k \in 2^V 
    \Bigg \}
\end{equation*}
for all $X \subseteq T$ is submodular on its own. 
The following example shows that this is not the case.
Consider the graph depicted in~\Cref{fig:counterexample_submodular}. 

\tikzstyle{nodeL}=[circle,draw, minimum size=4pt, inner sep = 0]
\tikzstyle{nodeI}=[circle,draw, minimum size=3pt, inner sep = 0, fill]
\tikzstyle{edge}=[]
\begin{figure}
    \centering
    \begin{center}
    \begin{tikzpicture}
        \node[nodeL, label ={$s_1$}] (s1) at (0, 1) {};
        \node[nodeL, label ={$s_2$}] (s2) at (0, 0) [left] {};
        \node[nodeI] (v)  at (2, 1) {};
        \node[nodeL, label = {$t$}] (t)  at (4, 0) [right]{};
        
        \draw[edge] (s1) -- node[above] {$1$} (v); 
        \draw[edge] (v) -- node[above] {$1$}(t) ;  
        \draw[edge] (s2) -- node[above] {$1$} (t); 
    \end{tikzpicture}
\end{center}
    \caption{Graph for which the set function $\sigma$ is not submodular.}
    \label{fig:counterexample_submodular}
\end{figure}
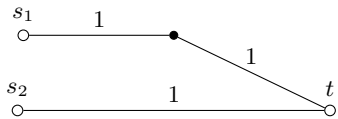

Let the terminal node set be given by $T = \{s_1,s_2,t\}$ with entry node set~$\smash{T^+=\{s_1,s_2\}}$ and exit node~$\smash{T^-=\{t\}}$. 
Assume that for every arc $a\in A$, we have $\mu_a x_a = 1$.
Then,
\begin{equation*}
    \sigma(\{s_1\}) + \sigma(\{s_2\}) = \frac{1}{\sqrt[r]{2}} + 1  < 2 
\end{equation*}
but $\sigma(\{s_1,s_2\}) + \sigma(\emptyset) = 2$, and 
therefore $\sigma(\{s_1\}) + \sigma(\{s_2\}) < \sigma(\{s_1,s_2\}) + \sigma(\emptyset)$ which shows that $\sigma$ is not submodular.


\section{Computational Results}
\label{sec:ComputationalResults}

We performed computational experiments with the derived inequalities on gas networks with $r = 2$, where we concentrate on the single source, single sink case.
The code was implemented in C and uses SCIP 9.1.1~\cite{SCIP9}, which performs linear programming based spatial branch-and-bound.
We use CPLEX 12.10 to solve the LPs and IPOPT 3.14.17~\cite{WaeB06} for the NLPs. We include the acyclicity constraints from~\cite{HabP21}.
The computations were obtained single-threaded in exclusive mode on a Linux cluster with 
Intel Xeon E5-1620 CPUs running at 3.5 Ghz with 32 GB memory.
Our results are shown in Table \ref{tab2}; TL refers to reaching the time limit of 8 hours.

We tested the impact of the inequality \eqref{b-cut} on several instances with different parameters for the maximum pressure and the cut depth with the $(s,t)$-balance $d$ being fixed. 
We created path instances \textsc{Path27} and \textsc{Path52} consisting of 27 and 52 nodes respectively.
Between each pair of nodes there are three different pipe options to be built. 
Furthermore, we introduced a super source and sink to \textsc{GasLib40} and \textsc{GasLib134} to use them for our tests.
Moreover, \textsc{GasLib40large} contains more diameter options, whereas for \textsc{GasLib40} only the original arcs can be built.
We considered the GasLib instances as passive instances by deleting active elements, like valves and switches, and by considering compressors as normal pipes. 
The costs $c\in\mathbb R_{>0}^A$ depend on the diameter and length of the arc and are computed according to Reuß et al. \cite{Reuss2019}. 

The results align with the structure the inequalities, meaning that \eqref{b-cut} is especially strong if the length $k$ of the disjoint cuts and the balance $d$ are high as well as the maximum pressure $\bar{\pi}$ is low.
We found that the maximum possible value for~$k$ performed best on path instances, which is why we do not include results with lower $k$.
On the other instances, usually a higher $k$ performed better. 
For \textsc{Path27} and \textsc{Path52}  with the tested settings, the inequality describes the integral hull resulting in the instances being solved at the root node.
Overall, we conclude that the addition of the inequalities boosts performance. It is especially helpful on the large \textsc{GasLib134} and \textsc{Path52}, where the gap is huge without it. Note that for \textsc{GasLib40large} the version with the inequalities performs worse because an optimal solution is found very late.
In future experiments, we will investigate the behavior of primal heuristics and adapting $k$ to the instance.

\begin{table}[tb]
\scriptsize
\centering
\caption{Computational results with \eqref{b-cut} for the $(s,t)$-case}\label{tab2}
\vspace{-1.5ex}
\begin{tabular}{@{}l@{\quad}r@{\quad}r@{\quad}r@{\quad}r@{\quad}r@{\quad}r@{}}
\toprule
instance & \eqref{b-cut} & $\bar{\pi}$ &time (s)& \# nodes & gap \% & $k$\\
\midrule
\textsc{Path27} & yes &70 & 0.7 & 1 & 0 & 26\\
\textsc{Path27} & no &70 & 21789 & 2056431 & 0 & --\\
\textsc{Path27} & yes &100 & TL & 1321823 & 8.6 & 26\\
\textsc{Path27} & no &100 & TL & 837490 & 40.9 & --\\
\textsc{Path27} & yes &200 & 6758 & 1701253 & 0 & 26\\
\textsc{Path27} & no &200 & TL & 1700438 & 0.45 & --\\
\textsc{Path52} & yes &98 & 0.1 & 1 & 0 & 51\\
\textsc{Path52} & no &98 & TL & 539108 & 94 & --\\
\textsc{GasLib40} & yes &100 & 18 & 453 & 0 & 12\\
\textsc{GasLib40} & no &100 & 17 & 854 & 0 & --\\
\textsc{GasLib40large} & yes &100 & 19735 & 2058996 & 0 & 12\\
\textsc{GasLib40large} & no &100 & 280 & 8484 & 0 & --\\
\textsc{GasLib134} & yes &70 & TL & 469862 & 41.5 & 40\\
\textsc{GasLib134} & no &70 & TL & 386046 & 114.7 & --\\
\bottomrule
\end{tabular}
\end{table}

   \begin{credits}
       \subsubsection{\ackname}
       This work is funded by the DFG via the SFB TRR 154 subprojects A01, A007, and A09.
       \subsubsection{\discintname}
       The authors have no competing interests to declare that are relevant to the content of this article.
   \end{credits}
   
    \newpage
    \bibliographystyle{splncs04}
    \bibliography{literature}

    \newpage

\appendix

\section{Appendix}

Recall that $S_1, \dots, S_k \in 2^V$ are \emph{disjoint $X$-cuts} for $X \subseteq T$ if $S_1\subseteq \dots \subseteq S_k$ induce $k$ disjoint $(X\cap T^+,T^-\setminus X)$-cuts.

\subsection{Details for the proof of Claim~\ref{claim:sigma_k} in the proof of \Cref{thm:separation}}
\label{app:claim1}

\begin{lemma}\label{lem:lattice}
        Let $k \in \N$. Then,
    \begin{equation*}
        \mathcal X _k \coloneqq \left \{X\subseteq T  \suchthat 
             \text{there are disjoint $X$-cuts } S_1, \dots, S_k \in 2^V 
            \right \}
    \end{equation*}
        is closed under set union and set intersection.
\end{lemma}
\begin{proof}
    Let $X$, $Y \in \mathcal X_k$. 
    To show that $ X \cup Y \in \mathcal X_k$ and that $ X \cap Y \in \mathcal X_k$, 
    we prove the existence of $k$ disjoint $X\cup Y$-cuts, as well as $k$ disjoint $X\cap Y$-cuts.
    
    Since $X$, $Y \in \mathcal X_k$ there exist sets $S^X_1, \dots, S^X_k \subseteq V$ which induce $k$ disjoint cuts with respect to $X$ and sets $S^Y_1, \dots, S^Y_k \subseteq V$  for $Y$. 
    We claim that the sets
    \begin{equation*}
        S^X_1\cup S^Y_1 \subseteq \dots \subseteq S^X_k\cup S^Y_k \subseteq V 
        \quad \text{and} \quad 
        S^X_1\cap S^Y_1 \subseteq \dots \subseteq S^X_k\cap S^Y_k \subseteq V
    \end{equation*}
    induce the required disjoint $X\cup Y$-cuts and $X\cap Y$-cuts, respectively.

    Clearly, $(X\cup Y)\cap T^+  \subseteq (S^X_1\cup S^Y_1)$ and $(X\cap Y)\cap T^+  \subseteq (S^X_1\cap S^Y_1)$ because $X\cap T^+  \subseteq S^X_1$ and $Y \cap T^+  \subseteq S^Y_1$.
    On the other hand, since $(T^-\setminus X)\subseteq (V\setminus S^X_k)$ and $(T^-\setminus Y)\subseteq (V\setminus S^Y_k)$, it holds that
    \begin{equation*}
        T^-\setminus (X \cup Y) = (T^-\setminus X )\cap (T^-\setminus Y ) \subseteq (V\setminus S^X_k) \cap (V\setminus S^Y_k) = V\setminus (S^X_k \cup S^Y_k).
    \end{equation*}
    Similarly, it follows that $(T^-\setminus (X \cap Y)) \subseteq (V\setminus (S^X_k \cap S^Y_k))$.
    This proves that the sets indeed induce cuts with respect to $X\cup Y$ and  $X\cap Y$, respectively.

    To show that the induced cuts are disjoint, let $a=(u,v) \in \delta(S^X_i\cup S^Y_i)$.
    We prove that $a\notin \delta(S^X_j\cup S^Y_j)$ for $j\neq i$.
     Assume without loss of generality that $i<j$, $u\in S^X_i \cup S^Y_i $, and $v\notin S^X_i \cup S^Y_i $.
    Since $S^X_i \subseteq S^X_j $ and $S^Y_i \subseteq S^Y_j$ it follows that $u \in S^X_j \cup S^Y_j$.
    Given that $\delta(S^X_i) \cap \delta(S^X_j) = \emptyset$ and $\delta(S^Y_i) \cap \delta(S^Y_j) = \emptyset$, it must hold that 
    $v\in S^X_j$ or $v\in S^Y_j$, or equivalently $v\in S^X_j \cup S^Y_j$.
    Thus, $a\notin \delta(S^X_j \cup S^Y_j)$.
    
    For the intersection, let $a=(u,v) \in \delta(S^X_i\cap S^Y_i)$ and assume without loss of generality that $i<j$, $u\in S^X_i \cap S^Y_i $, and $v\notin S^X_i \cap S^Y_i $.
    Again, since $i<j$ we have $u \in S^X_j \cap S^Y_j$.
    Using the fact that $\delta(S^X_i)\cap \delta(S^X_j)=\emptyset$ and $\delta(S^Y_i)\cap \delta(S^Y_j) = \emptyset$, we deduct that $v\in S^Y_j$ and $v\in S^Y_j$, or simply $v\in S^X_j \cap S^Y_j$.
    Consequently,  $a\notin \delta(S^X_j \cap S^Y_j)$.
\qed \end{proof}

\subsection{Details for the proof of Claim~\ref{claim:sigma_eval_poly} in the proof of \Cref{thm:separation}}
\label{app:claim2}

We consider an undirected graph $G=(V,E,w)$ with non-negative edge weights $w\in\Rp^E$.
The $k$ \emph{disjoint $(s,t)$-cut problem} finds $(s,t)$-cuts $S_1,\dots,S_k\in 2^V$ with $\delta(S_i)\cap \delta(S_j)=\emptyset$
for $i\neq j$ minimizing the total weight of the edges crossing the cuts, i.e., $\sum_{i=1}^k\sum_{e\in\delta(S_i) }w_e.$

This problem is solved in~\cite{Wagner1990} in a more general setting.
Let $D=(V,A,w)$ be a directed graph with non-negative arc weights $w\in\Rp^A$. 
An $(s,t)$-cut $S$ is an $(s,t)$-\emph{dicut} if the arcs crossing the cut are only forward arcs, i.e., $\delta(S)= \delta^+(S)$.
The $k$-\emph{dicut} problem is the computation of $k$ disjoint $(s,t)$-dicuts minimizing the total weight over all such sets.
It is shown in~\cite{Wagner1990} that the $k$-dicut problem contains the
$k$ disjoint $(s,t)$-cut problem.
The idea is to insert for every edge $e=(u,u') \in E$ two extra nodes $v_1$ and $v_2$ and replace $e$ with four arcs $a_1=(u,v_1)$, $a_2=(u',v_2)$ of weight $w_e$ and $a_3=(u',v_1)$, $a_4=(u,v_2)$ of zero weight.

We therefore describe in the following, how the $k$-dicut problem can be solved for an directed graph $D=(V,A,w)$ with $s$, $t\in V$ and $s\neq t$.
First note that instead of finding the cuts $S_1,\dots,S_k$, one can equivalently find the disjoint arc sets $\delta(S_1),\dots,\delta(S_k)$ minimizing the total weight. 
Assume without loss of generality that there exists a directed $(s,t)$-path $P$ of length $k$ each arc having zero weight in $D$.
Let $n$ be the number of vertices and $m$ the number of arcs in $D$ and let $M\in \R^{(m-n+1)\times m}$ be a fundamental cut matrix of $D$.
Consider the linear program
\begin{equation}
    \begin{aligned}
        \min_{y\in \R^A}       & \quad&& w^ty &  &&& \\
        \text{s.t.} &      && My =0, &  &&\\
                    &      && 0\leq y_a \leq 1 & \quad &&&a \in A\setminus P,\\
                    &      &&  y_a  = 1 & \quad &&& a\in P.\\
    \end{aligned}
    \label{min_k_cut}
\end{equation}
Since $M$ is totally unimodular, it follows that every extreme point solution $y$ to~\eqref{min_k_cut} fulfills $y_a\in\{0,1\}$ for every arc $a\in A$.
An optimal solution $y\in \{0,1\}^A$ indicates with $y_a=1$ if an arc $a\in A$ is an arc crossing one of the $k$-dicuts or not, i.e., $y_a=0$.
It is shown in~\cite{Wagner1990} how the vector $y$ can be transformed into the arc sets  $\delta(S_1),\dots,\delta(S_k)$.
For our application, that is to evaluate $\sigma_k(X)$ for $X\subseteq T$, it suffices to compute the optimal value $w^ty$ of~\eqref{min_k_cut}.

\end{document}